\documentclass[12pt,reqno,oneside]{amsart}
\usepackage{amsmath,amsthm,amsfonts,amssymb,graphicx}
\usepackage[mathscr]{eucal}
\usepackage{bbm}
\usepackage{indentfirst}
\usepackage{url}

\theoremstyle{plain}
\newtheorem{teo}{Theorem}[section]

\newtheorem{lem}[teo]{Lemma}
\newtheorem{prop}[teo]{Proposition}

\theoremstyle{definition}
\newtheorem{defn}[teo]{Definition}

\newtheorem{rem}[teo]{Remark}

\numberwithin{equation}{section}

\def\bb1{{\mathbbm{1}}}

\begin{document}
	\baselineskip=22pt
	\title[Frog model on $\mathbb{Z}$ with random survival parameter]{Frog model on $\mathbb{Z}$ with random survival parameter}
	\author{Gustavo~O.~de Carvalho}
	\author{F\'abio~P.~Machado}
	\address[F\'abio~P.~Machado]
	{Institute of Mathematics and Statistics
		\\ University of S\~ao Paulo \\ Rua do Mat\~ao 1010, CEP
		05508-090, S\~ao Paulo, SP, Brazil - fmachado@ime.usp.br}
	\noindent
	\address[Gustavo~O.~de Carvalho]
	{Institute of Mathematics and Statistics
		\\ University of S\~ao Paulo \\ Rua do Mat\~ao 1010, CEP
		05508-090, S\~ao Paulo, SP, Brazil - gustavoodc@ime.usp.br}
	\noindent
	\thanks{Research supported by CAPES (88887.676435/2022-00), FAPESP (23/13453-5)}
	
	\date{\today}

\begin{abstract}
We study the frog model on \( \mathbb{Z} \) with geometric lifetimes, introducing a random survival parameter. Active and inactive particles are placed at the vertices of \( \mathbb{Z} \). The lifetime of each active particle follows a geometric random variable with parameter \( 1-p \), where \( p \) is randomly sampled from a distribution \( \pi \). Each active particle performs a simple random walk on \( \mathbb{Z} \) until it dies, activating any inactive particles it encounters along its path. In contrast to the usual case where \( p \) is fixed, we show that there exist non-trivial distributions \( \pi \) for which the model survives with positive probability. More specifically, for $\pi\sim Beta(\alpha,\beta)$, we establish the existence of a critical value \( \beta=0.5 \), that separates almost sure extinction from survival with positive probability. Furthermore, we show that the model is recurrent whenever it survives with positive probability.
\end{abstract}

\maketitle

\section{Introduction and results}

We study the frog model, a branching random walk process often associated with the spread of a rumor or a disease in a population. Initially, a random number of particles are placed at each vertex of a connected graph $\mathcal{G}$, where one vertex is designated as the root. At time 0, particles at the root are active, while all others are inactive. Active particles perform a simple symmetric random walk on $\mathcal{G}$, activating any inactive particles they encounter along their path.

At every time step, each active particle may die with probability $1-p$ before jumping,
in which case it is removed from the model. Consequently, the lifetime of each active particle follows a geometric distribution with parameter $1-p$. The case where particles never die can be incorporated by setting $p=1$.

There has been considerable interest in identifying conditions under which the probability of survival is positive, i.e., ensuring that at all times at least one active particle remains. In \cite{phase_transition}, the authors proved that the process undergoes a phase transition in the lattice $\mathbb{Z}^d$ for $d \geq 2$ and in $\mathbb{T}^d$, the $d$-homogeneous tree, for $d \geq 3$, whereas in $\mathbb{Z}$, they showed that, under mild conditions on the initial number of particles per vertex, there is no phase transition. 

Following this initial work, several studies have further investigated $p_c$, the critical parameter, in graphs where a phase transition occurs \cite{phase_transition2,phase_transition3,phase_transition4,phase_transition5}.  
In \cite{frog_model_z}, the authors show that the frog model can survive in $\mathbb{Z}$ with positive probability, provided that either the random walks have a drift or the survival parameter $p=p_x$ varies appropriately with the initial position $x\in \mathbb{Z}$ of each particle.  

In this paper, we construct a particle-wise disorder setup to achieve survival with positive probability in $\mathbb{Z}$. This serves as an illustration of how introducing randomness into a parameter or the environment can significantly alter the properties of a model.

Other properties related to the frog model on $\mathbb{Z}$ have also been studied \cite{z_drift,z_coexistence,z_explosion}. In \cite{recorrencia_drift, frog_model_z}, the authors analyze the frog model on $\mathbb{Z}$ with drift and other variations, primarily focusing on recurrence versus transience, i.e., whether there are particles returning to the root infinitely often. In this paper, we further prove that, given survival, the frog model on $\mathbb{Z}$ is recurrent.

We now formally define the frog model. Since some of our results are not restricted to the geometric lifetime case, we start by defining the model with a general lifetime distribution. Let $\mathbb{N}=\{1,2,3,\dots\}$ and define $\mathbb{N}_0=\mathbb{N}\cup\{0\}$.  
Let $\{\eta_x\}_{x\in \mathbb{Z}}$ and $\{L_{x,i}\}_{x\in\mathbb{Z},i\in \mathbb{N}}$ be independent collections of i.i.d. random variables. Let $\{(S^{x,i}_n)_{n\in \mathbb{N}_0}\}_{x\in \mathbb{Z},i\in \mathbb{N}}$, with $S^{x,i}_0=x$, be a collection of independent simple symmetric random walks on $\mathbb{Z}$.  
	
For each $x\in \mathbb{Z}$, $\eta_x$ denotes the initial number of particles at vertex $x$.  
For each particle $i$ ($i\in\{1,\dots,\eta_x\}$) starting at vertex $x\in \mathbb{Z}$, $(S^{x,i}_n)_{n\in\mathbb{N}_0}$ denotes its random walk and $L_{x,i}$ is its lifetime.

We denote by $\text{FM}(\mathbb{Z},L,\eta)$ the frog model on $\mathbb{Z}$, where $L$ relates to the lifetime of each particle and $\eta$ determines the initial number distribution of particles per vertex.  
	
\begin{defn}
	A realization of the frog model is said to survive if at every time step there is at least one active particle. Otherwise, it dies out.
\end{defn}
	
It is also useful to describe the survival of the frog model in terms of an oriented percolation model. Let  
\[
A(x):= \{y \in \mathbb{Z}: S^{x,i}_n = y \text{ for some } i \in \{1, ..., \eta_x\} \text{ and } n \in \{0, ..., L_{x,i}\}\}
\]
be the set of vertices that can potentially be visited by particles initially placed at $x$. The vertices in $A(x)$ will effectively be visited if $x$ is visited and its particles are activated.  
	
Therefore, a realization of the frog model survives if and only if there exists an infinite sequence of distinct vertices $0=x_0,x_1,x_2,\dots$ in $\mathbb{Z}$ such that $x_j \in A(x_{j-1})$ for all $j \in \mathbb{N}$. 

We now introduce some random variables that will be useful for setting results on survival or extinction. Let
\begin{equation}\label{eq:d->}
	D_{z,i}^{\rightarrow}:=\max_{n\in \{0,...,L_{z,i}\}}\{S^{z,i}_n-z\}
	\hspace{.5cm}\text{ and } \hspace{.5cm}
	D_{z,i}^{\leftarrow}:=\max_{n\in \{0,...,L_{z,i}\}}\{z-S^{z,i}_n\}
\end{equation}
represent the maximum distance reached by the $i$-th particle starting at $z$ in the rightward (positive) and leftward (negative) directions, respectively.  

Additionally, define
\[
D^*_{z,i}:=\max_{n\in \{0,...,L_{z,i}\}}\{|S^{z,i}_n-z|\}=D_{z,i}^{\rightarrow}\vee D_{z,i}^{\leftarrow}
\]
as the maximum displacement of the $i$-th particle from $z$ in either direction.  

Both $\{D^\rightarrow_{z,i}\}_{z\in \mathbb{Z},i\in\mathbb{N}}$ and $\{D^*_{z,i}\}_{z\in \mathbb{Z},i\in\mathbb{N}}$ form collections of i.i.d. random variables. We denote their representative random variables by $D^\rightarrow$ and $D^*$, respectively.

\begin{prop}\label{prop:bz13} The following items hold:
\begin{itemize}
    \item[(i)] If $E(\eta)<\infty$ and $\limsup_{n\to \infty}nP(D^*\geq n)<\frac{1}{2E(\eta)}$, then
\[P(\text{FM}(\mathbb{Z},L,\eta)\text{ survives})=0.\]
    \item[(ii)] If $\liminf_{n\to \infty}nP(D^{\rightarrow}\geq n)>1/E(\eta)$ (where $1/E(\eta):=0$ when $E(\eta)=\infty$), then
\[P(\text{FM}(\mathbb{Z},L,\eta)\text{ survives})>0.\]
\end{itemize}
\end{prop}

Even though Proposition \ref{prop:bz13} brings a general overview of survival and extinction, it may be hard to determine the limits above for some $L$. Next, we focus on the geometric case and we add randomness to the survival parameter of each particle. Let $\{\pi_{x,i}\}_{x\in\mathbb{Z},i\in\mathbb{N}}$ be a collection of i.i.d. random variables assuming values in $(0,1]$. The geometric case is considered by setting
\[
P(L_{x,i}=k|\pi_{x,i}=p)=(1-p)p^k\mathbbm{1}_{(k\in\mathbb{N}_0)},
\]
which, for shortness, will be denoted as $L=\mathcal{G}_\pi$. For any $p\in(0,1]$, we also write $L=\mathcal{G}_p$ to denote the constant case $P(\pi=p)=1$.

Let $a\vee b$ denote $\max\{a,b\}$. Under the mild condition $E(\log(1\vee\eta))<\infty$, \cite[Theorem 1.1]{phase_transition} shows that for any $p\in(0,1)$, 
\begin{equation}\label{eq:amp02}
	P(\text{FM}(\mathbb{Z},\mathcal{G}_p,\eta)\text{ survives})=0.
\end{equation}

The previous result assumes that $\pi$ is almost surely constant. Introducing randomness in $\pi$ may cause this almost sure extinction to fail in some cases, which is our main goal. We first present trivial survival results for certain distributions $\pi$.

If $P(\pi=1)>0$, it is straightforward to see that $P(\text{FM}(\mathbb{Z},\mathcal{G}_\pi,\eta)\text{ survives})>0$ whenever $P(\eta=0)<1$, as each particle has a chance to live forever and perform an infinite random walk on $\mathbb{Z}$. Moreover, if there exists $a<1$ such that $P(\pi\in[0,a])=1$, then $P(\text{FM}(\mathbb{Z},\mathcal{G}_\pi, \eta)\text{ survives})=0$ whenever $E(\log(1\vee \eta))<\infty$. This follows from coupling the processes in such a way that $\text{FM}(\mathbb{Z},\mathcal{G}_\pi,\eta)$ is dominated by $\text{FM}(\mathbb{Z},\mathcal{G}_a,\eta)$ and applying (\ref{eq:amp02}) with $p =a$.

Among distributions $\pi$ that do not exhibit the aforementioned trivialities, we focus in particular on the Beta family. 

\begin{defn}
	We say that $\pi\sim Beta(\alpha,\beta)$, where $\alpha>0$ and $\beta>0$, if its density function is given by
	\[
	f_{\pi}(x)=\frac{x^{\alpha-1}(1-x)^{\beta-1}}{B(\alpha,\beta)}\mathbbm{1}_{(x\in (0,1))},
	\]
	where $B(\alpha,\beta)=\int_{0}^1 x^{\alpha-1}(1-x)^{\beta-1}dx$ is its normalizing constant.
\end{defn}

This family encompasses a wide variety of distributions in $(0,1)$ that are mathematically well-suited for our calculations. However, it is worth noting that survival results can be easily extended through couplings to stochastically larger (or smaller, in the case of almost sure extinction) non-Beta distributions.

We now state theorems describing the survival and extinction of the frog model as a function of 
$\alpha$ and $\beta$.

\begin{teo}\label{teo:>1}
    Consider that $\pi\sim Beta(\alpha,\beta)$ and $E(\eta)<\infty$. If $\alpha>0$ and $\beta>0.5$, then
    \[P(\text{FM}(\mathbb{Z},\mathcal{G}_\pi,\eta)\text{ survives})=0.\]
\end{teo}

\begin{teo}\label{teo:<0.5}
    Consider that $\pi\sim Beta(\alpha,\beta)$ and $P(\eta=0)<1$. Then,
    \[P(\text{FM}(\mathbb{Z},\mathcal{G}_\pi,\eta)\text{ survives})>0\]
    in any of the following cases:
    \begin{itemize}
        \item $\alpha>0$ and $\beta<0.5$
        \item $\alpha>\alpha_0$ and $\beta=0.5$, where $\alpha_0=\alpha_0(E(\eta))\geq 0$ is a function defined for $E(\eta)\in (0,+\infty)\cup\{+\infty\}$ given by
    \end{itemize}
    \[\alpha_0(E(\eta)):=
        \inf\{\alpha>0:B(\alpha,0.5)<E(\eta)\sqrt{2}\}.\]
\end{teo}

Figure \ref{fig:diagram} summarizes the results of Theorems \ref{teo:>1} and \ref{teo:<0.5} in a phase diagram for cases where $\eta$ satisfies $P(\eta=0)<1$ and $E(\eta)<\infty$. Note that in Theorem \ref{teo:<0.5}, $\alpha_0=0$ when $E(\eta)=\infty$, which makes the diagram complete in this case.

\begin{figure}[h]
\centering
\includegraphics[width=15cm]{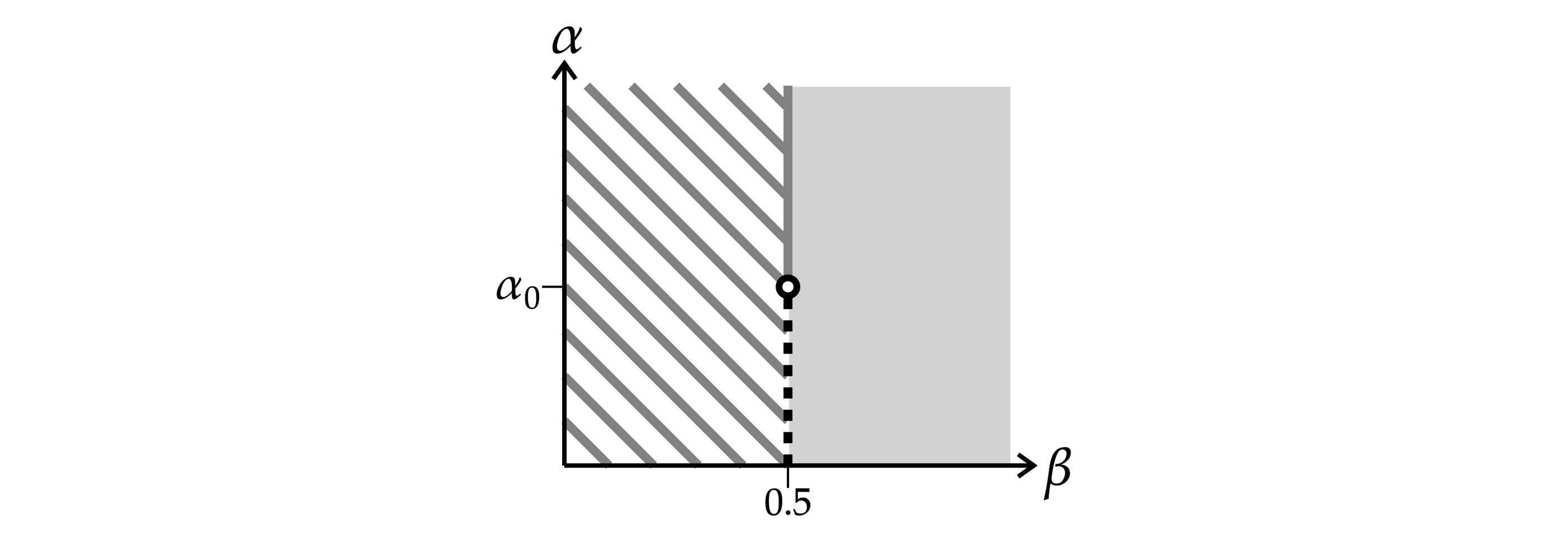}
\caption{Phase diagram for $\text{FM}(\mathbb{Z},\mathcal{G}_\pi,\eta)$ where $\pi\sim Beta(\alpha,\beta)$ and $\eta$ is such that $P(\eta=0)<1$ and $E(\eta)<\infty$. The shaded part represents the region where $P(\text{FM}(\mathbb{Z},\mathcal{G}_\pi,\eta) \text{ survives})=0$ while the hatched part represents the region where $P(\text{FM}(\mathbb{Z},\mathcal{G}_\pi,\eta) \text{ survives})>0$. The black circle and the dashed line represent the unknown parts of the diagram.}\label{fig:diagram}
\end{figure}

We now give special attention to distributions for which $P(\text{FM}(\mathbb{Z},L,\eta)\text{ survives})>0$, for a general $L$. In such cases, we can also analyze the recurrence of the model.

\begin{defn}
    Let $V_0(\text{FM}(\mathbb{Z},L,\eta$)) be the number of times vertex 0 is visited in $\text{FM}(\mathbb{Z},L,\eta)$. We say that $\text{FM}(\mathbb{Z},L,\eta$) is recurrent if
    \[P(V_0(\text{FM}(\mathbb{Z},L,\eta))=\infty)>0.\]
\end{defn}

It is straightforward that 
\[
P(V_0(\text{FM}(\mathbb{Z},L,\eta))=\infty) \leq P(\text{FM}(\mathbb{Z},L ,\eta)\text{ survives})
\]
and that the model is non-recurrent whenever 
 $
P(\text{FM}(\mathbb{Z},L,\eta)\text{ survives}) = 0$.

A natural question is whether $\text{FM}(\mathbb{Z},L,\eta)$ is recurrent whenever $P(\text{FM}(\mathbb{Z},L,\eta)\text{ survives}) > 0$. The next theorem provides a general answer to this question, highlighting the strong relationship between survival and recurrence.

\begin{teo}\label{teo:recor}
If $E(\eta)<\infty$ and $P(\text{FM}(\mathbb{Z},L,\eta) \text{ survives})>0$, then
    \[P(V_0(\text{FM}(\mathbb{Z},L,\eta))=\infty|\text{FM}(\mathbb{Z},L,\eta) \text{ survives})=1.\]
\end{teo}

\begin{rem}
The proof of Theorem \ref{teo:recor} relies on couplings and previous results on closely related rumor processes.  
Indeed, as another application, the same arguments used in our proof can also be applied to the rumor process defined in \cite{rumour} to show that the rumor reaches vertex zero infinitely often, almost surely, given its survival.
\end{rem}

Since visiting vertex $0$ infinitely often is only possible when the model survives, Theorem \ref{teo:recor} also implies that if $E(\eta)<\infty$, then
\[P(V_0(\text{FM}(\mathbb{Z},L,\eta))=\infty)=P(\text{FM}(\mathbb{Z},L,\eta)\text{ survives}).\] 

The key idea of the proof is to identify necessary conditions on the distribution of $L$ so that 
$
P(\text{FM}(\mathbb{Z},L,\eta)\text{ survives})>0.
$ 
These conditions, in turn, imply that, with probability 1, there will be potential for infinitely many particles to return to the root. When survival occurs, this potential is equivalent to  
$
V_0(\text{FM}(\mathbb{Z},L,\eta))=\infty$.

We can also apply Theorem \ref{teo:recor} to the specific case $L=\mathcal{G}_\pi$ with $\pi\sim Beta(\alpha,\beta)$ discussed earlier. By doing this, we conclude that Figure \ref{fig:diagram} is also a phase diagram regarding recurrence.

The remainder of this paper is dedicated to proving Theorems \ref{teo:>1}, \ref{teo:<0.5}, and \ref{teo:recor}. Section \ref{sec:firework} introduces models that serve as useful couplings for the frog model, while Section \ref{sec:demo_z} contains the main proofs of the theorems.

\section{Proofs}

\subsection{Coupling with rumor processes}\label{sec:firework} 

We begin by defining models that will later be used in a coupling argument to assist in proving our theorems.

We consider the rumor model described in \cite{rumour}. Let $\{I_z\}_{z\in\mathbb{Z}}$ be a collection of i.i.d. non-negative random variables representing the radius within which each vertex propagates the rumor. We let $\mathcal{A}_0=\tilde{\mathcal{A}_0}:=\{0\}$ and define recursively for $n\in\mathbb{N}_0$:
\begin{equation}\label{eq:def_rumor}
	\tilde{\mathcal{A}}_{n+1}:=\bigcup_{u \in \tilde{\mathcal{A}}_n}\{z \in \mathbb{Z}\setminus \mathcal{A}_n : |z-u| \leq I_u\} \hspace{5mm}\text{ and } \hspace{5mm}\mathcal{A}_{n+1}:=\mathcal{A}_n\cup \tilde{\mathcal{A}}_{n+1}.
\end{equation}

Here, $\mathcal{A}_n$ represents the set of vertices that are aware of the rumor at time $n$, while $\tilde{\mathcal{A}}_n$ consists of vertices that first receive the rumor at time $n$ and propagate it at time $n+1$. Thus, each vertex $z\in \mathbb{Z}$, upon receiving the rumor at time $n$, attempts to propagate it at time $n+1$ to its neighbors that have not yet received it and are within a distance of at most $I_z$.  
We say that the rumor percolates if $|\cup_{n\geq 0}\mathcal{A}_n|= \infty$, meaning that the rumor reaches infinitely many vertices. The percolation event is also equivalent to $\cap_{n\geq 0}\{\tilde{\mathcal{A}}_n\neq \emptyset\}$, which ensures that at every time step at least one vertex continues propagating the rumor.

We refer to this model as the \textit{bi-directional firework process}. The usual (uni-directional) firework process \cite{firework} can be defined similarly by modifying, for $n\in\mathbb{N}_0$, the definition of $\tilde{\mathcal{A}}_{n+1}$ in (\ref{eq:def_rumor}) to  
\[
\tilde{\mathcal{A}}_{n+1}:=\bigcup_{u \in \tilde{\mathcal{A}}_n}\{z \in \mathbb{N}\setminus \mathcal{A}_n : u< z \leq u+I_u\}
\]
while keeping all other definitions unchanged. Thus, the uni-directional firework process behaves similarly to the bi-directional one, but spreading occurs in only one direction, and negative vertices play no role. We denote the firework and bi-directional firework processes by $\text{FW}(I)$ and $\text{BFW}(I)$, respectively, where $I$ is the radius random variable. We assume that $P(I=0)\in(0,1)$ to avoid trivialities.

The following lemma establishes results regarding percolation in both processes.

\begin{lem}\label{lem:fire_bifire} The following items hold:
\begin{itemize}
    \item[(i)] $P(\text{FW$(I)$ percolates})>0$ if and only if $ \sum_{n=1}^\infty \prod_{i=0}^n P(I\leq i)<\infty$.
    \item[(ii)] If $\sum_{n=1}^\infty \prod_{i=0}^n P(I\leq i)<\infty$, then $P(\text{BFW$(I)$ percolates})>0$.
    \item[(iii)] If $\sum_{n=1}^\infty \prod_{i=0}^n P(I\leq i)^2=\infty$, then $P(\text{BFW$(I)$ percolates})=0$.
\end{itemize}
\end{lem}

\begin{proof}

Item (i) corresponds to Theorem 2.1 in \cite{firework}. 
The proof of item (ii) follows directly from item (i) and the fact that 
\[
P(\text{FW}(I) \text{ percolates}) \leq P(\text{BFW}(I) \text{ percolates}),
\]
which arises from a simple coupling argument between these processes.

We now focus on proving item (iii).
Let BFW$^*$($I$) denote the \textit{mirrored bi-directional firework process} with radius variable $I$. 
In this variant, we consider a collection $\{I_x\}_{x\in \mathbb{N}_0}$ of i.i.d. non-negative random variables. 
These ranges are defined only for vertices in $\mathbb{N}_0$ and are extended to the negative side by setting $I_{-x} := I_x$ for every $x \in \mathbb{N}$, thus introducing symmetry into the process.

Let $\mathcal{A}^{BFW^*}_0 = \tilde{\mathcal{A}}_0^{BFW^*} := \{0\}$ and define recursively, for $n \in \mathbb{N}_0$,
\[
\tilde{\mathcal{A}}^{BFW^*}_{n+1} := \bigcup_{u \in \tilde{\mathcal{A}}_n}
\{ z \in \mathbb{Z} \setminus \mathcal{A}^{BFW^*}_n : |z - u| \leq I_u \},
\quad \text{and} \quad
\mathcal{A}^{BFW^*}_{n+1} := \mathcal{A}^{BFW^*}_n \cup \tilde{\mathcal{A}}^{BFW^*}_{n+1}.
\]

By symmetry, for every $n \in \mathbb{N}_0$, we have $\mathcal{A}_n^{BFW^*} = \{-x_n, \ldots, x_n\}$ for some $x_n \in \mathbb{N}_0$. 
Due to the mirror effect of BFW$^*$($I$), we may equivalently assume that positive vertices spread the rumor only to larger positive vertices (as in FW($I$)). 
Indeed, we can ignore any positive vertex $x \in \tilde{\mathcal{A}}_n^{BFW^*}$ attempting to transmit the rumor to a negative vertex $y$, since $-x \in \tilde{\mathcal{A}}_n^{BFW^*}$ will also transmit the rumor to $y$, as $I_{-x} = I_x$. 
Hence, BFW$^*$($I$) can be viewed as a composition of two identical copies of FW($I$), one on the positive side and the other on the negative side. Therefore,
\[
P(\text{BFW$^*$($I$) percolates}) = P(\text{FW($I$) percolates}).
\]

Now consider any random variable $I$ such that
\[\sum_{n=1}^\infty \prod_{i=0}^n P(I\leq i)^2=\infty.\]

Consider the random variable $I^*$ such that $P(I^*\leq i)=P(I\leq i)^2$. We construct a coupling between BFW($I$) and BFW$^*$($I^*$) such that, for every $x \in \mathbb{N}_0$,
\begin{equation}\label{eq:i_i+}
I^*_{-x} = I^*_x = \max\{I_{-x}, I_x\},
\end{equation}
which is consistent with the definition above, since it preserves both 
$P(I^* \leq i) = P(I \leq i)^2$ and the symmetry of BFW$^*$($I^*$).

Define for $z\in\mathbb{Z}$,
\[\mathcal{B}_z:=\{x \in \mathbb{Z} : |z-x| \leq I_x\}\]
and
\[\mathcal{B}_z^*:=\{x \in \mathbb{Z} : |z-x| \leq I^*_x\}.\]
Note that (\ref{eq:i_i+}) implies that $\mathcal{B}_z\subset \mathcal{B}^*_z$ for every $z\in\mathbb{Z}$.

If BFW($I$) percolates, there exists an infinite sequence of distinct vertices $0=z_0,z_1,z_2,...$ such that $z_{n+1}\in\mathcal{B}_{z_{n}}$ for every $n\in\mathbb{N}_0$. Since $\mathcal{B}_{z_n}\subset \mathcal{B}^*_{z_n}$ for every $n\in\mathbb{N}_0$, it follows immediately that percolation in BFW($I$) implies percolation in BFW$^*$($I^*$).

Therefore, 
\begin{equation}\label{eq:bfw_fw}
P(\text{BFW($I$) percolates})\leq P(\text{BFW$^*$($I^*$) percolates})=P(\text{FW($I^*$) percolates}).\end{equation}

So, the proof is complete by noticing that $P(\text{FW($I^*$) percolates})=0$ by item (i), since
\[\sum_{n=1}^\infty \prod_{i=0}^n P(I^*\leq i)=\sum_{n=1}^\infty \prod_{i=0}^n P(I\leq i)^2=\infty.\]
\end{proof}

\begin{rem}
    Proposition 3.1 in \cite{rumour} states that $P(\text{BFW$(I)$ percolates})>0$  if and only if $\sum_{n=1}^\infty \prod_{i=0}^n P(I\leq i)<\infty$. The available proof, however, appears incomplete. In our opinion the statement cannot be confirmed on the basis of the current exposition. Unfortunately, we have not been able to supply the missing arguments. Thus, we use a weaker result, as stated in Lemma \ref{lem:fire_bifire} (iii). If this stronger statement is eventually proven, the constant in Proposition \ref{prop:bz13} (i) would improve from $1/2E(\eta)$ to $1/E(\eta)$, but item (ii), as well as Theorems \ref{teo:>1}, \ref{teo:<0.5} and \ref{teo:recor}, would remain unchanged.
\end{rem}

In particular, we consider the case where $I$ is defined as a function of two other random variables. Let $\{N_z\}_{z\in \mathbb{Z}}$ and $\{R_{z,i}\}_{z\in\mathbb{Z},i\in\mathbb{N}}$ be independent collections of i.i.d. random variables, where $N_z\in \mathbb{N}_0$ denotes the number of spreaders at vertex $z$ and $R_{z,i}\geq 0$ denotes the radius of the $i$-th spreader at vertex $z$. We define $I_z$ as the total propagation radius of the rumor at vertex $z$, considering all its spreaders:
\begin{equation}\label{eq:i_n_r}
	I_z:=\max_{i\in\{1,...,N_z\}} R_{z,i}.
\end{equation}

In this setup, $P(I\leq i)$ can be directly determined from the distributions of $N$ and $R$. 

\begin{lem}\label{lem:liminf_limsup} Consider (\ref{eq:i_n_r}). The following items hold:
\begin{itemize}
    \item[(i)] If $E(N)<\infty$ and $\limsup_{n\to \infty} nP(R \geq n)<\frac{1}{2E(N)}$, then
\[P(\text{BFW$(I)$ percolates})=0.\]
    \item[(ii)] If $E(N)<\infty$ and $\limsup_{n\to \infty} nP(R \geq n)<1/E(N)$, then
\[P(\text{FW$(I)$ percolates})=0.\]
  \item[(iii)] If $\liminf_{n\to\infty} nP(R\geq n)>1/E(N)$ (where $1/E(N):=0$ when $E(N)=\infty$), then
  \[P(\text{FW$(I)$ percolates})>0 \hspace{3mm}\text{and} \hspace{3mm}P(\text{BFW$(I)$ percolates})>0.\]
\end{itemize}
\end{lem}

\begin{proof}
The proof of every statement regarding the firework process can be found in \cite[Proposition 3.2]{firework_stations}, and combines Lemma $\ref{lem:fire_bifire}$ (i) with Kummer's test for series convergence. The extension to the bi-directional version follows from analogous arguments and calculations. In fact, one can also extend the results from the firework process to the bi-directional case without further calculations. This is achieved by noting that, by the same reasoning as in the proof of Lemma \ref{lem:fire_bifire} and inequality (\ref{eq:bfw_fw}), 
\begin{equation}\label{eq:inclusao}
\{\text{FW}(I)\text{ percolates}\}\subset\{\text{BFW}(I)\text{ percolates}\}\subset \{\text{FW}(I^*)\text{ percolates}\},\end{equation} when we consider a coupling in which, for $z\in\mathbb{Z}$, \[I_z=\max_{i\in\{1,...,N_z\}}R_{z,i}\qquad\text{and}\qquad I^*_z=\max\{I_{-z},I_z\}\overset{d}{=}\max_{i\in\{1,...,N^*_z\}}R_{z,i}\] with $N^*_0=N_0$ and $N^*_z=N_{-z}+N_z$ for $z\in \mathbb{Z}\setminus\{0\}$. 

The bi-directional part of item (iii) comes directly from the uni-lateral part of the same item together with the first inclusion in (\ref{eq:inclusao}). Item (i) follows from item (ii) together with the second inclusion in (\ref{eq:inclusao}), since $E(N^*_x)=2E(N_x)$ for $x\in\mathbb{N}$, and we may simply add more particles at vertex $0$ (which only helps percolation) to make $E(N^*_0)=2E(N_0)$.
\end{proof}

We now make use of comparisons between the frog model and both the firework and the bi-directional firework processes to prove Proposition \ref{prop:bz13}.

\begin{proof}[Proof of Proposition \ref{prop:bz13}]

We establish a connection between the frog model and FW$(I)$ and BFW$(I)$, both defined with $I$ as a function of $N$ and $R$, as in (\ref{eq:i_n_r}).  

On the one hand, when $N_z=\eta_z$ and $R_{z,i}=D_{z,i}^{*}$, BFW$(I)$ dominates $\text{FM}(\mathbb{Z},L,\eta)$, meaning that  
\[
P(\text{BFW$(I)$ percolates})\geq P(\text{FM}(\mathbb{Z},L,\eta)\text{ survives}).
\]  
This dominance holds because we overestimate the number of vertices visited by each particle's random walk. This relationship, together with Lemma $\ref{lem:liminf_limsup}$ (i), prove Proposition \ref{prop:bz13} (i).

On the other hand, when $N_z=\eta_z$ and $R_{z,i}=D_{z,i}^{\rightarrow}$, FW$(I)$ is dominated by $\text{FM}(\mathbb{Z},L,\eta)$, meaning that  
\[
P(\text{FW$(I)$ percolates})\leq P(\text{FM}(\mathbb{Z},L,\eta)\text{ survives}).
\]  
This dominance holds because we consider only the rightward portion of each particle's random walk. This relationship, together with Lemma $\ref{lem:liminf_limsup}$ (iii), prove Proposition \ref{prop:bz13} (ii).
\end{proof}

To apply Proposition \ref{prop:bz13}, we need estimates for the distribution of $D^\rightarrow$. In the geometric case, we use the following useful result.

\begin{lem}[\cite{improved_upper_bound}, Lemma 2.1]\label{lem:dist_d} Consider $\text{FM}(\mathbb{Z},L,\eta)$ with $L=\mathcal{G}_\pi$. For every $z\in\mathbb{Z}$, $i\in\mathbb{N}$ and $p\in (0,1)$,
    \[P(D_{z,i}^\rightarrow \geq n|\pi_{z,i}=p)=\bigg(\frac{1-\sqrt{1-p^2}}{p}\bigg)^n.\]
\end{lem}

\subsection{Proof of Theorems} \label{sec:demo_z}

We begin this section by stating an auxiliary lemma, whose proof follows from elementary calculus and is therefore omitted.

\begin{lem}\label{lema:euler}
For every $1<\gamma<2$, the following items hold
\begin{itemize}
    \item[(i)] $\lim_{n\to\infty}{(1-{n^{-\gamma}})^{-n}}=1$
    \item[(ii)] $\lim_{n \to \infty}n(1-n^{-\gamma/2})^n=0$
    \item[(iii)] $\lim_{n\to\infty}n\frac{(1-{n^{-\gamma/2}})^n
}{{(1-{n^{-\gamma}})^n}}=0$
\end{itemize}
\end{lem}

We now proceed to the proofs of the theorems.

\begin{proof}[Proof of Theorem \ref{teo:>1}] 

Note that $\{D^*_{x,i}\geq n \}=\{D^{\rightarrow}_{x,i}\geq n\}\cup \{D^{\leftarrow}_ {x,i}\geq n \}$ and that $D^{\rightarrow}_{x,i}$ and $D^{\leftarrow}_{x,i}$ are identically distributed. Therefore, 
\begin{equation}\label{eq:d*/2}
P(D^{\leftarrow}_{x,i}\geq n)=P(D^{\rightarrow}_{x,i}\geq n)\geq \frac{1}{2}P(D^*_{x,i}\geq n ).
\end{equation}

So, the proof is complete by Proposition $\ref{prop:bz13}$ (i) if we show that
\[\lim_{n\to \infty}nP(D^\rightarrow \geq n)=0.\]

Recall that $P(D_{x,i}^\rightarrow \geq n|\pi_{x,i}=p)=(\frac{1-\sqrt{1-p^2}}{p})^n$ by Lemma \ref{lem:dist_d} and note that $\sqrt{1-x^2}=\sqrt{(1+x)(1-x)}\geq \sqrt{1-x}$ for $x\in(0,1)$. Therefore, for any $z\in(0,1)$,
\begin{equation}\label{eq:2_integr}
\begin{aligned}
nP(D^\rightarrow \geq n)
&\leq \int_{0}^z n\bigg(\frac{1-\sqrt{1-x}}{x}\bigg)^n\frac{x^{\alpha-1}(1-x)^{\beta-1}}{B(\alpha,\beta)}dx\\
&\hspace{5mm}+\int_{z}^1 n\bigg(\frac{1-\sqrt{1-x}}{x}\bigg)^n\frac{x^{\alpha-1}(1-x)^{\beta-1}}{B(\alpha,\beta)}dx
\end{aligned}
\end{equation}

The idea is to choose $z=z(n)\in(0,1)$ such that we are able to show that both integrals in (\ref{eq:2_integr}) go to 0 when $n\to \infty$. Note that $(\max\{\frac{1}{\beta},1\},2)$ is a non-empty interval as $\beta>0.5$ and let $z:=1-\frac{1}{n^{\gamma}}$ with $\gamma\in (\max\{\frac{1}{\beta},1\},2)$.

Since $\frac{1-\sqrt{1-x}}{x}$ is an increasing function of $x\in(0,1)$, we use Lemma \ref{lema:euler} (iii) to conclude that
\[\begin{aligned}
\int_{0}^z n\bigg(\frac{1-\sqrt{1-x}}{x}\bigg)^n\frac{x^{\alpha-1}(1-x)^{\beta-1}dx}{B(\alpha,\beta)}dx &\leq n\bigg(\frac{1-\sqrt{1-z}}{z}\bigg)^n \int_{0}^z\frac{x^{\alpha-1}(1-x)^{\beta-1}dx}{B(\alpha,\beta)}dx\\
&\leq n\bigg(\frac{1-\sqrt{1-z}}{z}\bigg)^n\\
&= n\frac{(1-\frac{1}{n^{\gamma/2}})^n
}{{(1-\frac{1}{n^\gamma})^n}}\rightarrow 0.
\end{aligned}\]

We have that $\big(\frac{1-\sqrt{1-x}}{x}\big)^n x^{\alpha-1}\leq x^{\alpha-1-n}\leq z^{\alpha-1-n}$ for any $x\in(z,1)$ and $n>\alpha-1$. Moreover, $1-\gamma\beta<0$ as $\gamma\in(\max\{\frac{1}{\beta},1\},2)$. Therefore, we take a sufficiently large $n$ and use Lemma \ref{lema:euler} (i) to conclude that
\[\begin{aligned}
\int_{z}^1 n\bigg(\frac{1-\sqrt{1-x}}{x}\bigg)^n\frac{x^{\alpha-1}(1-x)^{\beta-1}}{B(\alpha,\beta)}dx&\leq\frac{nz^{\alpha-1-n}}{B(\alpha,\beta)} \int_{z}^1 (1-x)^{\beta-1}dx\\
&=\frac{nz^{\alpha-1-n}}{B(\alpha,\beta)} \frac{(1-z)^\beta}{\beta}\\
&=\frac{n^{1-\gamma\beta}(1-\frac{1}{n^\gamma})^{\alpha-1-n}}{B(\alpha,\beta)\beta}\rightarrow 0,
\end{aligned}\]
completing the proof.
\end{proof}

\begin{proof}[Proof of Theorem \ref{teo:<0.5}]

Let $c_1>c_2>...>c_k\geq 0$ be a sequence of constants and $z_i=z_i(n):=1-\frac{c_i}{n^2}$ for $i\in\{1,2,...,k\}$. Let
\[I_{n,\alpha,\beta}(i):=n\int_{z_i}^{z_{i+1}} \bigg(\frac{1-\sqrt{1-x^2}}{x}\bigg)^n\frac{x^{\alpha-1}(1-x)^{\beta-1}}{B(\alpha,\beta)}dx.\]

By Lemma \ref{lem:dist_d} and the superadditivity property of the limit inferior,
\begin{equation}\label{eq:liminf_np}
\liminf_{n\to \infty} nP(D^\rightarrow \geq n)\geq \liminf_{n\to \infty}\sum_{i=1}^{k-1}I_{n,\alpha,\beta}(i) \geq \sum_{i=1}^{k-1}\liminf_{n\to \infty} I_{n,\alpha,\beta}(i).
\end{equation}

For any $0<z\leq x\leq 1$ and $n\geq\alpha-1$, we have that $$\bigg(\frac{1-\sqrt{1-x^2}}{x}\bigg)^nx^{\alpha-1}\geq \frac{\big(1-\sqrt{(1-z)(1+z)}\big)^n}{x^{n+1-\alpha}}\geq \big(1-\sqrt{2(1-z)}\big)^n.$$ So, for sufficiently large $n$,
\begin{equation}
\begin{aligned}\label{eq:i_int}
I_{n,\alpha,\beta}(i)&\geq \frac{n\big(1-\sqrt{2(1-z_i)}\big)^n}{B(\alpha,\beta)}\int_{z_i}^{z_{i+1}} (1-x)^{\beta-1}dx\\
&= \frac{n\big(1-\sqrt{2(1-z_i)}\big)^n}{B(\alpha,\beta)}\bigg[\frac{(1-z_i)^\beta}{\beta}-\frac{(1-z_{i+1})^\beta}{\beta}\bigg]\\
&\geq \frac{(c_i^\beta-c_{i+1}^\beta)}{B(\alpha,\beta)\beta}n^{1-2\beta}\bigg(1-\frac{\sqrt{2c_i}}{n}\bigg)^n.
\end{aligned}\end{equation}

When $\beta<0.5$, we have by $(\ref{eq:i_int})$ that $\liminf_{n\to \infty}I_{n,\alpha,\beta}(1)=\infty$ for any choice of $c_1>c_2\geq 0$. We combine (\ref{eq:liminf_np}) and Proposition \ref{prop:bz13} (ii) to conclude that $P(\text{FM}(\mathbb{Z},\mathcal{G}_\pi,\eta)\text{ survives})>0$ for any value of $E(\eta)\neq 0$.

Now we deal with the case $\beta=0.5$. By (\ref{eq:i_int}), we have that
\begin{equation}\label{eq:liminf_i}
\liminf_{n\to\infty} I_{n,\alpha,0.5}(i)\geq \frac{2(\sqrt{c_i}-\sqrt{c_{i+1}})}{B(\alpha,0.5)}e^{-\sqrt{2c_i}}.
\end{equation}

Note that $B(\alpha,0.5)=\int_0^1x^{\alpha-1}(1-x)^{-0.5}dx$ is decreasing in $\alpha$, and that $\lim_{\alpha\to \infty}B(\alpha,0.5)=0$ as an immediate consequence of the Dominated Convergence Theorem. This implies that $\alpha_0(E(\eta)):=
\inf\{\alpha>0:B(\alpha,0.5)<E(\eta)\sqrt{2}\}$ is well-defined for $E(\eta)\neq 0$ (it is the infimum of a non-empty set). For all $\alpha>\alpha_0$, this also implies that $B(\alpha,0.5)<E(\eta)\sqrt{2}$ and that there exists $\gamma=\gamma(\alpha)<1$ such that \begin{equation}\label{ineq:b_0.5}
B(\alpha,0.5)<\gamma E(\eta)\sqrt{2}.\end{equation}

Note that
\[\lim_{m\to\infty}\frac{1}{m}\sum_{i=1}^\infty e^{-\frac{i\sqrt{2}}{m}}=\lim_{m\to\infty}\frac{e^{-\frac{\sqrt{2}}{m}}}{m\big(1-e^{-\frac{\sqrt{2}}{m}}\big)}=\frac{1}{\sqrt{2}},\]
where $\lim_{m\to\infty} m(1-e^{-\frac{\sqrt{2}}{m}})=\sqrt{2}$ can be calculated by the L'Hospital rule. So, as $\gamma(\alpha)<1$, there exists $m$ and $k$ such that 
\begin{equation}\label{eq:ineq_k_m}
\frac{1}{m}\sum_{i=1}^{k-1} e^{-\frac{i\sqrt{2}}{m}}>\frac{\gamma}{\sqrt{2}}.\end{equation}

Consider such $m$ and $k$ and take $c_i=(\frac{k-i}{m})^2$ for $i\in \{1,2,...,k\}$. We combine (\ref{eq:liminf_i}), (\ref{ineq:b_0.5}) and (\ref{eq:ineq_k_m}) to obtain for $\alpha>\alpha_0$ that
\[\sum_{i=1}^{k-1}\liminf_{n\to \infty}I_{n,\alpha,0.5}(i)\geq\frac{2}{B(\alpha,0.5)}\frac{1}{m}\sum_{i=1}^{k-1} e^{-\frac{i\sqrt{2}}{m}}> \frac{1}{E(\eta)},\]
where $\frac{1}{E(\eta)}:=0$ when $E(\eta)=\infty$. Again, we use (\ref{eq:liminf_np}) and Proposition \ref{prop:bz13} (ii) to conclude that $P(\text{FM}(\mathbb{Z},\mathcal{G}_\pi,\eta)\text{ survives})>0$ for $\beta=0.5$ and $\alpha>\alpha_0$.
\end{proof}

\begin{proof}[Proof of Theorem \ref{teo:recor}]

Since $E(\eta)<\infty$ and $P(\text{FM}(\mathbb{Z},L,\eta)]\text{ survives})>0$ by assumption, we have by Proposition \ref{prop:bz13} (i) that
\[\limsup_{n\to \infty} n P(D^{*}\geq n)\geq\frac{1}{2E(\eta)}.\]

Therefore, there exists an infinite increasing sequence $(a_i)_{i\in\mathbb{N}}$ such that $a_i P(D^{*}\geq a_i)\geq \frac{1}{4E(\eta )}$ for all $i\in\mathbb{N}$. It is possible to create a subsequence $(b_i)_{i\in\mathbb{N}}$ such that, in addition to $b_i P(D^{*}\geq b_i)\geq \frac{1}{4E( \eta)}$, we also have that $b_{i+1}\geq 2 b_{i}$ for all $i\in\mathbb{N}$. Note that $P(D^*\geq j)\geq P(D^*\geq b_{i+1})\geq\frac{1}{4E(\eta)b_{i+1}}$ holds for every $j\in \{b_{i}+1,...,b_{i+1}\}$, since $P(D^*\geq j)$ is non-increasing. So, we have that
\begin{equation}\label{eq:infinity}
    \begin{aligned}
\sum_{n=1}^{\infty} P(D^{*}\geq n)&\geq \sum_{i=1}^{\infty}\sum_{j=b_{i}+1}^{b_{i+1}}P(D^{*}\geq j)\\
&\geq \sum_{i=1}^{\infty}\frac{(b_{i+1}-b_i)}{4E(\eta)b_{i+1}}\\
&\geq \sum_{i=1}^{\infty}\frac{1}{8E(\eta)}=\infty.
\end{aligned}\end{equation}

We use ($\ref{eq:d*/2}$), ($\ref{eq:infinity}$) and the independence between $\eta_x$ and $D^\leftarrow_{x,1}$ to conclude that
\[\sum_{x=1}^{\infty} P(\eta_x\geq 1, D_{x,1}^{\leftarrow }\geq x)\geq \frac{P(\eta\geq 1)}{2}\sum_{x=1}^{\infty} P(D_{x,1}^{* }\geq x)=+\infty\]
and, by the Borel-Cantelli lemma, that
\begin{equation}\label{eq:nd_io}
P(\eta_x\geq 1 \text{ and } D_{x,1}^{\leftarrow }\geq x \text{ for infinitely many } x\in\mathbb{N})=1.
\end{equation}

Under the event (\ref{eq:nd_io}), there exist infinitely many vertices in $\mathbb{N}$ that contain at least one particle capable of taking enough steps to the left to revisit vertex 0.  

Thus, given that the frog model survives on the positive side (i.e., every vertex $x\in\mathbb{N}$ is visited by some particle), it follows that  
\[
V_0(\text{FM}(\mathbb{Z},L,\eta))=\infty
\]
with probability 1.  

Moreover, by (\ref{eq:nd_io}) and symmetry, we have that  
\[
P(\eta_{x}\geq 1 \text{ and } D_{x,1}^{\rightarrow }\geq -x \text{ for infinitely many } x\in \mathbb{Z}\setminus\mathbb{N}_0)=1.
\]
Similarly to the previous case, given that the frog model survives on the negative side (i.e., every vertex $x\in \mathbb{Z}\setminus\mathbb{N}_0$ is visited by some particle), it follows that  
\[
V_0(\text{FM}(\mathbb{Z},L,\eta))=\infty
\]
with probability 1.  

This completes the proof, as the survival of the frog model implies survival on either the positive or the negative side.

\end{proof}

\section*{Acknowledgements}
We thank the anonymous referees for their comments and criticisms, which greatly helped us improve the manuscript. This study was financed in part by the Coordenação de Aperfeiçoamento de Pessoal de Nível Superior - Brasil (CAPES) - Finance Code 001. It was also partially supported by FAPESP (grant 2023/13453-5).

\bibliographystyle{alpha}
\bibliography{sample}

@article{rumour,
      title={How fast do rumours spread?}, 
      author={R. Roy and K. Saha},
      year={2024},
      volume={191},
      number={130},
      journal={Journal of Statistical Physics}
}

@article{firework,
 author = {V. V. Junior and F. P. Machado and M. Zuluaga},
 journal = {Journal of Applied Probability},
 number = {3},
 pages = {624--636},
 publisher = {Applied Probability Trust},
 title = {Rumor processes on $\mathbb{N}$},
 volume = {48},
 year = {2011}
}

@article{firework_stations,
author = {D. Bertacchi and F. Zucca},
year = {2013},
pages = {486–511},
title = {Rumor Processes in Random Environment on $\mathbb{N}$ and on Galton-Watson Trees},
volume = {153},
number = {3},
journal = {Journal of Statistical Physics}
}

@article{frog_model_z,
author = {D. Bertacchi and F. P. Machado and F. Zucca},
year = {2014},
pages = {256-278},
title = {Local and global survival for nonhomogeneous random walk systems on $\mathbb{Z}$},
volume = {46},
number = {1},
journal = {Advances in Applied Probability}
}

@article{phase_transition,
author = {O. S. M. Alves and F. P. Machado and S. Popov},
title = {{Phase Transition for the Frog Model}},
volume = {7},
journal = {Electronic Journal of Probability},
number = {16},
publisher = {Institute of Mathematical Statistics and Bernoulli Society},
pages = {1 - 21},
year = {2002}
}

@article{phase_transition2,
title={The critical probability for the frog model is not a monotonic function of the graph},
volume={41},
number={1},
journal={Journal of Applied Probability},
publisher={Cambridge University Press}, 
author={L. R. Fontes and F. P. Machado and A. Sarkar}, year={2004}, pages={292–298}}

@article{phase_transition3,
author = {E. Lebensztayn and J. Utria},
year = {2019},
pages = {169–179},
volume={176},
journal = {Journal of Statistical Physics},
title = {A new upper bound for the critical probability of the frog model on homogeneous trees}
}

@article{phase_transition4,
title={Phase transition for the frog model on biregular trees}, 
author={E. Lebensztayn and J. Utria},
year={2020},
journal = {Markov Processes And Related Fields},
volume={26},
pages={447-466},
number={3}
}

@article{phase_transition5,
title={Critical Parameter of the Frog Model on Homogeneous Trees with Geometric Lifetime}, 
author={S. Gallo and C. Pena},
year={2022},
journal = {Journal of Statistical Physics},
volume={190},
number={34}
}

@article{recorrencia_drift,
author = {N. Gantert and P. Schmidt},
title = {{Recurrence for the frog model with drift on $\mathbb{Z}$}},
volume = {15},
journal = {Markov Processes and Related Fields},
number = {1},
pages = {51 -- 58},
year = {2009},
}

@article{z_drift,
author = {A. Ghosh and S. Noren and A. Roitershtein},
title = {{On the range of the transient frog model on $\mathbb{Z}$}},
volume = {49},
journal = {Advances in Applied Probability},
number = {2},
pages = {327-343},
year = {2017}
}

@article{improved_upper_bound,
author = {E. Lebensztayn and F. P. Machado and S. Popov},
title = {An Improved Upper Bound for the Critical Probability of the Frog Model on Homogeneous Trees},
journal = {Journal of Statistical Physics},
volume = {119},
number = {1},
pages = {331–345},
year = {2005},
publisher = {Taylor & Francis},
}

@article{z_coexistence,
author = {M. Holmes and D. Kious},
title = {Coexistence of lazy frogs on $\mathbb{Z}$},
journal = {Journal of Applied Probability},
volume = {59},
number = {3},
pages = {702–713},
year = {2022}
}

@article{z_explosion,
author = {V. Bezborodov and L. {Di Persio} and P. Kuchling},
title = {Explosion and non-explosion for the continuous-time frog model},
journal = {Stochastic Processes and their Applications},
volume = {171},
pages = {104329},
year = {2024}
}

\end{document}